\documentclass[amsmath,amssymb,twoside,11pt]{article}
\usepackage{amsfonts, amssymb}

\usepackage{ amsthm}
\usepackage[fleqn]{mathtools}
\usepackage{enumerate}
\usepackage{multicol}
\usepackage{graphicx}
\usepackage{enumerate}
\usepackage{fancyhdr}
\usepackage{color,latexsym,amsmath,amsfonts,amssymb}
\usepackage[utf8]{inputenc}
\usepackage{layout,verbatim,setspace,cite}
\usepackage{authblk}

\usepackage{mathptmx}
\usepackage{graphicx}
\usepackage{subfig}
\usepackage{url}
\newtheorem{theorem}{Theorem}[section]

\newtheorem{proposition}[theorem]{Proposition}

\newtheorem{definition}[theorem]{Definition}
\newtheorem{remark}[theorem]{Remark}
\newtheorem{example}[theorem]{Example}

\def\mf{\mathfrak}
\def\F{\mathbb{F}}

\newcommand{\weakstar}{{\,\,\raise5pt\hbox{$\underrightarrow{\mathit{{}_{weak^*}}}$}\,\,}}

\newcommand{\erre}{\mbox{$\mathbb{R}$}}
\newcommand{\Fs}{\mathcal{F}}
\newcommand{\As}{\mathcal{A}}
\newcommand{\Eb}{\mathbb{E}}

\providecommand{\keywords}[1]{\textbf{{Key Words:}} #1}
\providecommand{\msc}[1]{\textbf{{M.S.C.:}} #1}

\title{ }
\begin{document}
\date{}
\title{\bf Set-valued Brownian motion}
\author{Domenico Candeloro,  Coenraad C.A. Labuschagne,    Valeria Marraffa, Anna Rita Sambucini}
\newcommand{\Addresses}{{
  \bigskip
 D.~Candeloro, \textsc{Department of Mathematics and Computer Sciences, University of Perugia,
 1, Via Vanvitelli -- 06123, Perugia (Italy)}\par\nopagebreak
 \textit{E-mail address}, D.~Candeloro: \texttt{domenico.candeloro@unipg.it}

  \medskip
Coenraad C.A. ~ Labuschagne, \textsc{
Programme in Quantitative Finance,
Department of Finance and Investment Management.
         University of Johannesburg
         P O Box 524, Auckland Park 2006 (South Africa)}\par\nopagebreak
 \textit{E-mail address}, C.C.A. ~ Labuschagne
\texttt{Coenraad.Labuschagne@gmail.com}\\

\medskip
Valeria Marraffa,\textsc{
Department of Mathematics and Computer Sciences,
University of Palermo,
 Via Archirafi 34, 90123 Palermo (Italy)}\par\nopagebreak
\textit{E-mail address}, V. ~ Marraffa \texttt{ valeria.marraffa@unipa.it}\\

\medskip
A.R. ~Sambucini, \textsc{Department of Mathematics and Computer Sciences, University of Perugia,
 1, Via Vanvitelli -- 06123, Perugia (Italy)}\par\nopagebreak
\textit{E-mail address}, A.R.~Sambucini: \texttt{anna.sambucini@unipg.it}
}}
\maketitle

\begin{abstract}
Brownian motions, martingales, and Wiener processes are introduced and studied for set valued functions taking  values in the subfamily of compact convex subsets of arbitrary Banach space $X$. 
The present paper is an application of the paper \cite{L1} in which an embedding result is obtained which considers also the ordered structure of $ck(X)$ and of \cite{GL1,GL2} in which these processes are considered in f-algebras.  
\end{abstract}

\noindent
\msc{60J65, 58C06, 46A40}
\\
\noindent
\keywords{Brownian motion, R{\aa}dstr\"{om} embedding theorem, Vector Lattices, Marginal distributions, generalized Hukuhara difference. }

\section{Introduction}

It is well known that the concept of Brownian motion is one of the most important in probability theory and its applications.
 The starting point of the present research are the papers \cite{bck,Grobler,GroblerCor, GL1,GL2,LW20101,sg2007,s1994,Shreve} in which stochastic integration is studied in partially ordered spaces or in the fuzzy set valued case. The literature in this field is  rich, we can cite for example \cite{s1990,s1990a,vw2014,so,GLM,Grobler2,dallas,bms,cs2014,BCS,cs2015}.

Here the notion of set valued Brownian motion is introduced and studied for the case of compact convex subsets of a Banach space $X$.
The paper is organized as follows: in section 2 the basic properties of the hyperspace $ck(X)$ and its embedding in $C(K)$ are introduced.
Since, in order to properly define a set valued Browian motion, a difference and a multiplicative structure are needed, the embedding and the Riesz structure of $C(K)$ are used.
For this reason the theory of integration in vector lattice is very important and useful, see for example \cite{jafa,s1994,vw2012,vw2012e}.\\
In section 3 examples of $ck(X)$-valued Browian motion are given together with some properties and with some characterizations, similar to the usual ones, involving martingales and so on.
In section 4 a possible extension to arbitrary Banach lattices is given: this is  done in the more abstract framework of \cite{GL1,GL2}, with the purpose to compare the two types of construction in the particular case here discussed, where the Banach lattice is $C(K)$.

In the appendix a characterization of the generalized Hukuhara difference which extends \cite{ls2010} is introduced.

\section{Probability distributions}
We recall  from \cite[Chapter II]{cv1977} the following notations that will be used in the present paper.
Let $X$ be a Banach space
with its dual $X^*$
 and 
let  $ck(X)$ be the subfamily of $2^X \setminus \emptyset$ of all
 compact, convex subsets of $X$.\\
As in \cite{cv1977} for all $A,B \in ck(X)$ and $\lambda \in \mathbb{R}$  the Minkowski addition and scalar multiplication are defined as
\begin{eqnarray}\label{operations}
A + B = \{ a+b: a \in A, b \in B \}, \,\,\, \mbox{and} \,\,\,\, 
\lambda A =  \{ \lambda a : a \in A \}
\end{eqnarray}

\noindent Let $H$ be the corresponding Hausdorff metric on $ck(X)$, i.e.
$$
H(A,B)=\max(e_d(A,B),e_d(B,A))
$$
where the excess $e_d(A,B)$ of the set $A$ over the set $B$ is
defined as
$$
e_d(A,B)=\sup\{ d(a,B):a\in A \}=\sup   \{  \inf_{b\in B}d(a,b):
a\in A\}.
$$

\noindent It is known that the family  $ck(X)$ endowed with the Hausdorff metric is a complete metric space. For every $C \in  ck(X)$, the
{\it support function of} $C$ is denoted by $s( \cdot, C)$ and is
defined by $s(x^*, C) = \sup \{ \langle x^*,c \rangle : \ c \in C\}$
for each $x^* \in X^*$. \\
Clearly, the map $x^* \longmapsto s(x^*, C)$
is sublinear on $X^*$ and 
$$-s(-x^*, C) = \inf \{ \langle x^*,c \rangle : \ c \in C\}, \mbox{\,\, for each \,\,} x^* \in X^*.$$

The following theorem holds:
\begin{theorem}\label{th5.7diL1}{\rm (\cite[Theorem 5.7]{L1})}
Let $X$ be a Banach space; then there exist a compact Hausdorff
 space $K$ and a map $j: ck(X) \rightarrow C(K)$ such that
\begin{enumerate}[\rm \bf  (\ref{th5.7diL1}.a)]
\item $j(\alpha A + \beta C) = \alpha j(A) + \beta j(C)$ for all $A,C \in ck(X)$ and $\alpha, \beta \in \mathbb{R}^+$,
\item $d_H(A,C) = \| j(A) - j(C) \|_{\infty}$ for every $A,C \in ck(X)$,
\item $j(ck(X))$ is norm closed in $C(K)$,
\item $j(\mbox{co}(A \cup B) = \max \{j(A), j(C) \}$, for all $A,C \in ck(X)$.
\end{enumerate}
\end{theorem}

\noindent The R\aa dstr\"om embedding  $\widetilde{j(ck(X))}$ of $ck(X)$ is given by $j:  ck(X) \to\widetilde{j(ck(X))}$, where
 $ j(C)= s(\cdot, C) \mbox { for all } C\in ck(X) $
 and  $\widetilde {j( ck(X)) }$ is the closure of the span of $\{s(\cdot, C) : C\in   ck(X)\}$
 in $(C(B_{X^*}), \sigma(X^*,X))$.
 Here $C(B_{X^*}) =\{ f:B_{X^*} \to \mathbb R: f \mbox { is continuous} \}$,
 $B_{X^*}$ denotes the unit ball of $X^*$  and $\sigma(X^*,X)$ denotes the
 weak$^*$ topology on $X^*$.
The bounded-weak-star (bw*) topology is the strongest topology of  $B_{X^*}$ with coincides with the weak$^*$ topology
of $B_{X^*}$ on every ball $B^r_{X^*}:= \{ f \in B_{X^*} : \|f\| \leq r\}$.\\
Let  ($\frak{B}( ck(X))$ ) denote the Borel $\sigma$-algebra on  $(ck(X), d_H)$.
\\

In order to define Brownian multivalued  motion a multiplication and a difference in $ck(X)$ are needed.
For what concerns the difference see the Appendix 
(however we shall always consider the difference $B_1-B_2$ of two convex and compact sets as the element $j(B_1)-j(B_2)$ in $C(K)$), 
while
to access the averaging properties of conditional expectation operators a
multiplicative structure is needed. In the Riesz space setting the most natural
multiplicative structure is that of an f-algebra. This gives a multiplicative
structure that is compatible with the order and additive structures on the
space. The ideal, $E^e$, of $E$ generated by $e$, where $e$ is a weak order unit of
$E$ and $E$ is Dedekind complete, has a natural f-algebra structure. This is
constructed by setting $(Pe) \cdot (Qe) = PQe = (Qe) \cdot (Pe)$ for band projections
$P$ and $Q$, and extending to $E^e$ by use of Freudenthal's Theorem. In fact this
process extends the multiplicative structure to the universal completion $E^u$,
of $E$. This multiplication is associative, distributive and is positive in the
sense that if $x, y \in E^+$ then $xy > 0$. Here $e$ is the multiplicative unit.\\

Thus the multiplication operation $\cdot : ck(X) \times ck(X) \rightarrow C(K)$ can be defined by:
$$A \cdot B =j(A) \cdot j(B).$$
If $X$ is finite dimensional then 
 $j(B_X) \cdot j(B)= j(B)$, and $B_X \cdot B$ exists not only in $C(K)$ but also in $ck(X)$ (and of course coincides with $B$).

\section{$ck(X)$-valued  Brownian motion}

Now we shall introduce a Brownian motion taking values in the space $ck_r(X)$, where $X$ is any general Banach space. 
(Here the notation $ck_r(X)$ means all the indicator functions of the type $r1_B$, as $r$ varies in $\erre$ and $B$ in $ck(X)$).

In order to do this, let us denote by $e$ the unit function in $C(K)$. In case $X$ is finite-dimensional, $e=j(B_X)$, the corresponding element of the unit ball of $X$.

\begin{definition}\label{bm-ab}\rm
Let $S$ denote the hyperspace we are interested in, i.e. 
$ck_r(X)$, and
let $(B_t)_t$ be a process taking values in 
$S$,  namely for every $t \geq 0$  $B_t : \Omega \rightarrow  
S \subset C(K)$.
This process will be called {\em set-valued Brownian motion} if 
the following conditions are satisfied:
\begin{description}
\item[\ref{bm-ab}.1] There exists an $f$-algebra $L$ such that
$B_t(\omega)\in L$ for each $\omega\in \Omega$ and each $t>0$;
\item [\ref{bm-ab}.2] $B_t, B_t^2$ are $C(K)$-valued Bochner integrable functions for each $t > 0$;  
\item [\ref{bm-ab}.3] For every evaluation functional $f \in C(K)^*$, the process $f(B_t)_t$ is a standard real Brownian motion.
\end{description}
\noindent We recall that an {\em evaluation functional} $f$ associates to every $x\in C(K)$ the value $x(k)$ for some fixed $k\in K$.

\end{definition}
\begin{example}\label{we} \rm
The following is an example of a set-valued Brownian motion, when $X$ is finite-dimensional:
\( (B_t)_t = (W_t e )_t \)
where $(W_t)_t$ is the standard scalar Brownian motion, and $e$ is the unit ball in $X$.
Then for every $f \in C(K)^* $ such that $f(e)= 1$ it is
\[f (W_t e )_t = (W_t f(e))_t = (W_t)_t.\]
So for every elementary event $\omega$ such that $W_t(\omega)
 > 0$ the set $W_t e = W_t(\omega) j(B_X) \in j(S)$, while if $W_t(\omega) < 0$ the meaning of $B_t(\omega)$ is $-|W_t(\omega)|j(B_X)\in -j(S)$.



Next, for every  real number $t$ and every element $B\in ck(X)$, the notation $tB$ represents the indicator function $t1_B$.

Finally, if $(W_t)_{t>0}$ denotes the standard Brownian motion, 
and if we set $V_t:=W_te$ for each positive $t$, then we have
 shown that $(V_t)_{t>0}$ is a Brownian motion taking values in $S:=ck_r(X)$ (or in $j(S)$ after embedding).
\end{example} 

From now on, let $(\Omega,\As,P)$ denote any fixed probability space, with a $\sigma$-algebra $\As$ and a countably additive probability measure $P$. 
\begin{definition}\rm
Let $\Gamma: \Omega \to  ck(X)$ be a measurable function. Define
$$P_{\Gamma}(B)=P(\Gamma(\omega)\subset B) 
\mbox { for all } B\in \frak{B}( ck(X)) 
$$
and 
$$F_{\Gamma}(Y)=P(\Gamma(\omega) \subset Y)
 \mbox{ for all } Y\in ck(X).  $$
\end{definition}
Then $P_{\Gamma}: \frak{B}(ck(X)) \to [0,1]$ is a probability measure (the {\em probability distribution of $\Gamma$}), and $F_{\Gamma}: {\rm ck}(X) \to [0,1]$ is its  \emph{distribution function}.


\begin{proposition} 
Let $\Gamma: \Omega \to {\rm ck}(X)$ be a measurable set-function.
Then $F_{\Gamma}= F_{j\circ\Gamma} (j(\cdot))$. 
\end{proposition}
\begin{proof}
It is
\begin{eqnarray*}
F_{\Gamma} (Y) &=& P( \Gamma \subset Y) \,\,\,  \mbox{where} \,\,\, \Gamma \subset Y \Longleftrightarrow j(\Gamma) \leq j(Y) \\
F_{\Gamma} (Y) &=& P( \Gamma \subset Y) = P(j(\Gamma) \leq j(Y)) = F_{j(\Gamma)} (j(Y)).
\end{eqnarray*}
\end{proof}

\begin{example}
\rm Let us assume that $X_1$ and $X_2$ are two real-valued random variables, $X_1\leq X_2$, and consider the variable $\Gamma:=[X_1,X_2]$ taking values in the hyperspace $ck(\erre)$. 
Now, when $Y$ is an element of $ck(\erre)$, i.e. $Y=[y_1,y_2]$, the condition 
$\Gamma\subset Y$ means $[X_1\geq y_1,X_2\leq y_2]$, and so
$$F_{\Gamma}(Y)=P([y_1\leq X_1\leq X_2\leq y_2]).$$ 
On the other hand, in this situation, the {\em unit sphere} of the dual space of $\erre$ is simply the set $\{-1,1\}$, and, for every set $[a,b]\in ck(\erre)$, one has
$$s(x^*,[a,b])=\left\{\begin{array}{rr}
-a,& x^*=-1\\ & \\
\ b,& x^*=1.
\end{array}\right.$$
Hence, one can write $j([a,b])=(-a,b)$ as soon as $a,b\in \erre,\ a\leq b$. Then $j(\Gamma)=(-X_1,X_2)$ and $j(Y)=(-y_1,y_2)$: the condition $j(\Gamma)\leq j(Y)$ now means $-X_1\leq -y_1$ and $X_2\leq y_2$, and again one has
$$F_{j(\Gamma)}(j(Y)) = P([y_1\leq X_1\leq X_2\leq y_2]) = F_{\Gamma}(Y).$$

\noindent
Let $X_1, Z$ be two independent random variables with distribution $\Gamma(1,\lambda)$, and denote $X_2:=X+Z$.
 Then clearly $0\leq X_1\leq X_2$, and  $\underline{X}= [X_1,X_2]$ defines a $ck(\erre)$-valued variable.
\\
 In order to compute its distribution function, fix arbitrarily $y_1$ and $y_2$ in $\erre$, with $0\leq y_1\leq y_2$. Then 
\begin{eqnarray*}
F_{\underline{X}}([y_1,y_2]) &=& P([y_1\leq X_1\leq X_2\leq y_2])
=\int_{y_1}^{y_2}\left({\int_0^{y_2-x}}f_Z(z)dz\right) f_{X_1}(x)dx=\\
&=&\lambda^2\int_{y_1}^{y_2}{\int_0^{y_2-x}}e^{-\lambda x}e^{-\lambda z}dz dx.
\end{eqnarray*}
Simple computations give finally
\[F_{\underline{X}}([y_1,y_2])=e^{-\lambda y_1}-e^{-\lambda y_2}+\lambda(y_1-y_2)e^{-\lambda y_2}= 
F_{(-X_1,X_2)} (-y_1,y_2).\]
\\

\begin{figure} [h] 
\begin{center}
  \includegraphics [width=0.6\textwidth]{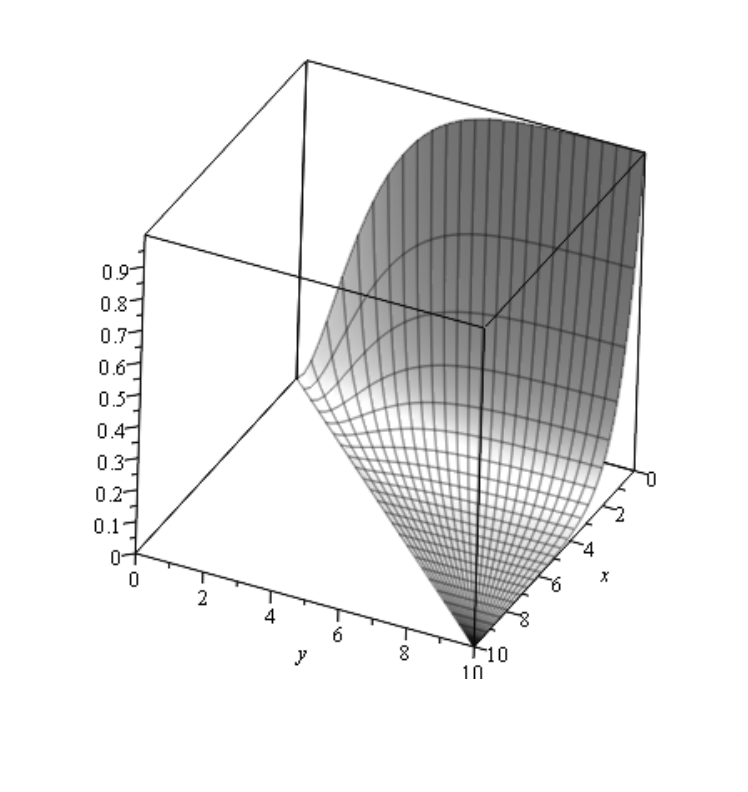} 
\caption{distribution function}
\end{center}
\end{figure}
\end{example}

\rm
In \cite{sg2007} the set valued Gaussian distribution is defined to satisfy the condition $F_{\Gamma}=F_{j\circ \Gamma}(j(\cdot))$.

 

\begin{theorem}\label{equiv}

Assume that 
$W_t: \Omega \to ck(X)$ is a weakly continuous $L$-valued function of $t\geq 0$
 that satisfies $W_0=0$. Moreover suppose that $W_t$ and $W^2_t$ are Bochner integrable for each $t$. Let $\{t_0, \ldots t_m\}$ be such that 
$0=t_0< t_1 < \dots <t_m$.

Then $(W_t)_t$ is a Brownian motion if and only if one of 
the following statements 
holds for any evaluation function $f\in C(K)^*$:
\begin{enumerate}[\rm \bf \ref{equiv}.i)]
\item The increments 
$f(W_{t_{1}}-W_{t_{0}}), f(W_{t_{2}}-W_{t_{1}}), \dots, f(W_{t_{m}}-W_{t_{m-1}})$ 
are independent and each of these increments is normally distributed with null mean and variance equal to $t_{i+1}-t_i$.
\item The random variables 
$f(W_{t_{1}}), f(W_{t_{2}}), \dots, f(W_{t_{m}})$
are jointly normally distributed with means equal to zero and co-variance matrix $V$  given by 
\begin{eqnarray*}
 V &=&
 \left( \begin{array}{cccc}
t_1 & t_1 & \cdots & t_1\\
t_1 & t_2 & \cdots & t_2\\
\vdots & \vdots & & \vdots \\
t_1 & t_m & \cdots & t_m
\end{array} \right);
\end{eqnarray*}

\item The random variables 
$f(W_{t_{1}}), f(W_{t_{2}}), \dots, f(W_{t_{m}})$
have the joint moment-generating function given by
\begin{eqnarray*}
\varphi (u_1, \ldots, u_m)&=& 
\exp \Big\{ \frac{1}{2} u^2_m (t_{m} - t_{m-1})\Big\}\cdots \exp 
\Big \{ (u_1 + u_2+ \ldots +u_{m})^2 t_1\Big\},
\end{eqnarray*}
for every $u_1, \ldots, u_m \in \mathbb{R}$.
\end{enumerate}
\end{theorem}
\begin{proof}
Let $\{t_0, \ldots t_m\}$ be fixed with 
$0=t_0< t_1 < \dots <t_m$ and consider 
an  evaluation function $f$. By  {\bf \ref{equiv}.i}) it is
\begin{eqnarray*}
 t_{i+1} - t_{i} &=&  {\rm Var} (f(W_{t_{i+1}}) - f(W_{t_{i}}) )
= {\rm Var} (f(W_{t_{i+1}} - W_{t_{i}}) )\\
&=& {\mathbb E} [(f(W_{t_{i+1}} - W_{t_{i}})^2]=
{\mathbb E}[ f^2 (W_{t_{i+1}} - W_{t_{i}})]=\\
&=&  {\mathbb E}[ f^2 (W_{t_{i+1}}) +f^2( W_{t_{i}}) - 2 f(W_{t_{i+1}}) \cdot f(W_{t_{i}}) ]=\\
&=& {\mathbb E} [f^2(W_{t_{i+1}}) - f^2(W_{t_{i}})] 
\end{eqnarray*}
Now, for every evaluation function $f$ the process $f(W_t)_t$ is a scalar Brownian motion and so all the three conditions are equivalent thanks to  \cite[Theorem 3.3.2]{Shreve} since  
\begin{eqnarray*}
V = 
\left( \begin{array}{cccc}
\mathbb E[f(W(t_1))^2] & \mathbb E[f(W(t_1))f(W(t_2))] & \cdots & \mathbb E[f(W(t_1))f( W(t_m))]\\
\mathbb E[f(W(t_2)) f(W(t_1))] & \mathbb E[f(W(t_2))^2] & \cdots & \mathbb E[f(W(t_2)) f(W(t_m))]\\
\vdots & \vdots & & \vdots \\
\mathbb E[f(W(t_m)) f(W(t_1))] & \mathbb E[f(W(t_m)) f(W(t_2))] & \cdots &  \mathbb E[f(W(t_m))^2]
\end{array} \right)
\end{eqnarray*}
and 
\begin{eqnarray*}
\varphi (u_1, \ldots, u_m)&=& \mathbb E\Big[ \exp \Big \{ u_m f( W(t_m)) + u_{m-1} f(W(t_{m-1}))+ \ldots  + u_{1} f(W(t_{1}))\Big\}\Big]\\
&=&  \mathbb E\Big[ \exp \Big \{u_m f( W(t_m) -W(t_{m-1}) ) + (u_{m-1} + u_m) f (W(t_{m-1}) -W(t_{m-2}) )\\
&+& \ldots \, + (u_1 + u_2+ \ldots +u_{m})f( W(t_1))\Big \}\Big] =
\end{eqnarray*}
\begin{eqnarray*}
&=& \mathbb E \Big[\exp \Big \{u_m f( W(t_m) -W(t_{m-1}) )\Big\}\Big]\cdot
\\
&\cdots & 
\mathbb E \Big[\exp \Big \{ (u_1 + u_2+ \ldots +u_{m}) f(W(t_1))\Big\}\Big]\\
&=& \exp \Big\{ \frac{1}{2} u^2_m (t_{m} - t_{m-1})\Big\}\cdots \exp \Big \{ (u_1 + u_2+ \ldots +u_{m})^2 t_1\Big\}.
\end{eqnarray*}
So, by Definition \ref{bm-ab},  $(W_t)_t$ is a Brownian motion.
\end{proof}

\begin{definition} \rm
For every Bochner integrable
set-valued function $W$, the conditional expectation $E(W|\Fs)$ of $W$ with respect to
a sub $\sigma$-algebra $\Fs \subset \As$ is a Bochner integrable function with respect to $(\Omega, \Fs, \lambda)$ such that for every $f \in C(K)^*$
 it is
$$ \Eb(f (W) | \Fs) =f(  \Eb(W|\Fs) ).$$
\end{definition}
\begin{definition}\label{mart} \rm
A set valued process $(M_t)_t$
 is a pointwise martingale if
\begin{enumerate}[\rm \ref{mart}.1)]
\item $M_t$ is Bochner integrable for every $t$;
\item $\Eb ( M_t | \Fs_s) = M_s$, for every $s < t$ where $(\Fs_s)_s$ is the natural filtration of $(M_t)_t$.
\end{enumerate}
\end{definition}

\begin{theorem}
Assume that $(B_t)_t$ is a set-valued Brownian motion, taking values in $L$.
Then whenever $0<s<t$ are fixed in $\erre$, one has
$$\Eb(B_t^2|\Fs_s)=B_s^2+(t-s)e.$$
\end{theorem}
\begin{proof} Let $f$ be any evaluation functional. Then we have
$$\Eb(f(B_t^2)|\Fs_s)=\Eb((f(B_t))^2|\Fs_s)=f(B_s^2)+t-s=f(B_s^2+(t-s)e)$$
by the usual properties of scalar Brownian motion and multiplicativity property of $f$. So, by arbitrariness of $f$, this leads to the assertion. 
\end{proof}

\rm
Clearly, this result means that, under the stated hypotheses, the sequence $(B_t^2-te)_t$ is a pointwise martingale.\\

The last theorem can be reversed, in some sense: more precisely,
\begin{theorem}
Let $(B_t)_t$ be a weak set-valued Gaussian process with homogeneous increments, such that $B_0=0$.
If $(B_t^2-te)_t$ is a pointwise martingale, then $(B_t)_t$ is a Wiener process (therefore, assuming also that the trajectories of $(B_t)_t$ are weakly continuous, one can conclude that $(B_t)_t$ is a set-valued Brownian motion). 
\end{theorem}
\begin{proof}
Indeed, from the martingale condition, one can deduce that $\Eb(B_t^2-te)$ is constant with respect to $t$, and therefore null, since $B_0=0$. So, $\Eb(B^2_t)=te$ for all $t$. Now, if $0<s<t$, thanks to the homogeneity property:
$$\Eb(f(2B_tB_s)=\Eb(f[B_t^2+B_s^2-(B_t-B_s)^2])=$$
$$=\Eb([f(B_t^2)+f(B_s)^2-f(B_{t-s}^2)])=t+s-t+s=2s$$
and this is precisely the defining property for a (weak) Wiener process. \end{proof}


Let $(M_t)_{t\geq 0}$ be an 
$L$-valued adapted process,
then $\sum_{i=0}^{n-1} [M_{t_{j+1}}- M_{t_{j}}]^2 \in C(K)$ for every partition $\pi=\{0=t_0<t_1<\dots<t_n=T\}$ of $[0,T], T > 0$.
\begin{definition}\label{2var}
The quadratic variation $[M_t, M_t]$ of an 
$L$-valued adapted process $(M_t)_t$, when it exists, 
is given by the following limit
$$\lim_{\|\pi\|\rightarrow 0} \| f([M_t, M_t](T) ) -  f( \sum_{i=0}^{n-1} [M_{t_{j+1}}- M_{t_{j}}]^2) \|_2 = 0$$
for every evaluation $f \in C(K)^*$ and every $T>0$.
\end{definition}

\begin{theorem} (Theorem of Levy)
Let $M_t$ be a martingale relative to a filtration $\mathcal F_t$ with $M_0=0$. Assume that $M_t$ has weakly continuous paths and $[M_t, M_t](T)=T$ for all $T\geq 0$. Then $(M_t)_t$ is a set-valued Brownian motion.
\end{theorem}

\begin{proof}
From the assumptions on  $M_t$, we get that, for each evaluation $f\in C(K)^*$,  $f(M_t)$ is a martingale with $f(M_0)=0$, $f(M_t)$ has continuous paths and $[f(M_t), f(M_t)](t)=t$ for all $t\geq 0$. Thus, $f(M_t)$ is a Brownian motion and so
$M_t$ is a set valued Brownian motion.
\end{proof}

\begin{definition}\label{is}
A set valued process $(W_t)_t$ is integrable with respect to a Brownian motion $(B_t)_t$ if
for every $T > 0$
there exists an element $I_T \in C(K)$ such that:
\begin{enumerate}[\rm \ref{is}.1)]
\item $(I_T)_T$ is a martingale with respect to $(B_t)_t$;
\item for every evaluation $f \in C(K)^*$ it is
\[ f(I_T) = (I) \int_0^T f(W_t) d (fB_t)\]
where the last integral is in the Ito sense.
\end{enumerate}
\end{definition}
For instance, the process $(B_t)_t$ is integrable, with $I_T=\frac{B_T-T}{2}$; more generally, if $(B_t)_t$ takes values in an $f$-algebra $L$, then the process $(B_t^k)_t$ is integrable for every positive integer $k$, and the usual Ito formula holds.

\section{Brownian motion in vector lattices}\rm
In this section we generalize the notions of Brownian Motion introduced before, replacing the space $C(K)$ with a particular Riesz space $E$ having an order unit $e$. 
\begin{definition}\label{defbm-G}\rm 
(\cite[Definition 3.6]{GL2})
Let $(B_t,\mf{F}_t)$ be an adapted stochastic process in the De\-de\-kind complete Riesz space $E$ with conditional expectation $\mathbb{F}$ and unit element $e$.  The process is called an $\mathbb{F}$-conditional Brownian motion in $E$ if for all $0\le s<t$ we have
\begin{enumerate}
\item[\bf (\ref{defbm-G}.1)] $B_0=0;$
\item[\bf(\ref{defbm-G}.2)] the increment $B_t-B_s$  is $\mathbb{F}$-conditionally independent of $\mf{F}_s$;
\item[\bf (\ref{defbm-G}.3)] $\F(B_t-B_s)=0;$
\item[\bf (\ref{defbm-G}.4)] $\F[(B_t-B_s)^2]=(t-s)e;$
\item[\bf (\ref{defbm-G}.5)] $\F((B_t-B_s)^4]=3(t-s)^2e.$
\end{enumerate}
\end{definition}

\begin{remark}\label{nota} \rm 
It was noted in\cite [page 901]{Grobler} that the definition of a Brownian motion in the Riesz space setting yields a Brownian motion in the classical case of real valued Brownian motion; i.e., a real valued stochastic process satisfies conditions (\ref{defbm-G}.1)-( \ref{defbm-G}.5) if and only if it is a Brownian motion.
\end{remark}

\begin{theorem}\label{final}
Let $B_t$ be a set valued stochastic process. Then $B_t$ is a Brownian motion if and only if for any evaluation function $f\in C(K)^*$ and every pair $(s,t)$ of positive real numbers with $s<t$:
\begin{enumerate}
\item[ \rm \bf (\ref{final}.1)] $f(B_0)=0;$
\item[ \rm \bf (\ref{final}.2)] the increment $f(B_t)-f(B_s)$  is $(\mathbb{F}f)$-conditionally independent of $f(\mf{F}_s)$;
\item[ \rm \bf (\ref{final}.3)] $\F(f(B_t)-f(B_s))=0;$
\item[ \rm \bf (\ref{final}.4)] $\F[(f(B_t)-f(B_s))^2]=(t-s)f(e);$
\item[ \rm \bf (\ref{final}.5)] $\F((f(B_t)-f(B_s))^4]=3(t-s)^2f(e).$ 
\end{enumerate} 
\end{theorem}

\begin{proof}
We note that $B_t$ is a Brownian motion if and only if $f(B_t)$ is a Brownian motion for each $f\in C(K)^*$ which is equivalent to the conditions  (\ref{final}.1) -  (\ref{final}.5) by the Remark \ref{nota}.
\end{proof}

\section{Appendix}
At the beginning of the paper we have claimed that, in order to introduce a notion of Brownian motion in this context, a kind of difference between sets is necessary. Here, 
following \cite{ls2010}, for every $A \in ck(X)$ let
$-A$ be  the opposite of the set $A$, namely $-A = \{-a: a \in A\}$ and consider the following difference between sets:
\begin{definition}\rm (\cite[Definition 1]{ls2010})
For every $A, B \in ck(X)$ the generalized Hukuhara difference of $A$ and $B$  (gH-difference for short), when exists, is the set $C \in ck(X)$ such that
\begin{eqnarray}\label{ghdiff}
 A \ominus_g B := C \Longleftrightarrow \left\{ \begin{array}{ll}
(i) & A = B + C, \,\,\,\,\, \mbox{or}\\
(ii) & B = A + (-C).
\end{array} \right. 
\end{eqnarray}
\end{definition}
\begin{remark}\label{osservaz-diff} \rm
By \cite[Propositions 1,6 and Remarks 2-5]{ls2010}   if the set $C$ exists it is unique and coincides with the Hukuhara difference between $A$ and $B$. Moreover a necessary condition for the existence is that either $A$ contains a traslate of $B$ or $B$ contains a traslate of $A$. 
If equations (\ref{ghdiff}.i) and (\ref{ghdiff}.ii) hold simultaneusly then $C$ is a singleton. Finally
\begin{enumerate}[\rm \bf \ref{osservaz-diff}.1)]
\item $A \ominus_g B \in ck(X)$ then $B \ominus_g A = - (A \ominus_g B)$;
\item $A \ominus_g A = \{ 0\}$,
\item $(A+B) \ominus_g B = A$,  $A \ominus_g (A -B) = B$, $A \ominus_g (A+B) = -B$.
\end{enumerate}
\end{remark}

If $A$ is compact  and convex subset of $X$ then it is characterized by its support function $s_A$ by Hahn-Banach theorem (see for example  \cite[Proposition II.16]{cv1977}). 
It is possible to express the gH-difference of convex compact sets using support functions.\\
 Given $A,B,C \in ck(X)$ let
$s(\cdot,A), s(\cdot,B), s(\cdot,C), 
s(\cdot,-C)$  be the support functions of $A,B,C,-C$ respectively.\\
Again by \cite[Propositions II-19]{cv1977} the map $A \mapsto s(\cdot,A)$ is injective, $s(x^*,A+B)=s(x^*,A) + s(x^*,B)$, $s(x^*,\lambda A)= \lambda s(x^*,A)$ for every non negative $\lambda$,
while
\begin{eqnarray*}
s(x^*,-A) &=& \sup \{<x^*,-x>, x \in A \} = \sup\{< -x^*,x>, x\in A\} =
\\ &=&  s(-x^*,A) \geq - s(x^*,A)
\end{eqnarray*}
And the equality in the last line holds when the opposite of $A$ is a set $C \in ck(X)$ such that $A + C = \{0\}$, namely $s(x^*,\ominus_g A) = -s(x^*,A)$.
So in  general  $s(-x^*,A) \geq - s(x^*,A)$ and the equality holds when  equation (\ref{ghdiff}.i)  holds.

We recall some well-known facts concerning Banach spaces.

\begin{theorem}{\rm \cite[Theorem 2.3]{gh1985}}\label{corollario2.5Gine}
Let $X$ be any Banach space and  
$H:B_{X^*} \to \erre$ be any mapping. Then $H$ is the support function of a convex compact subset of $X$ if and only if
$H$ is $bw^*$-continuous, subadditive and positively homogeneous. 
\end{theorem}
A consequence of this result can be stated in the following way.
\begin{proposition}
Let $(B_t)_t:=(W_te)_t$ be the example of Brownian motion in a finite-dimensional space $X$, given above in the Example \ref{we}.
Then, for each positive $t$, the function $j(B_t^2-\int_0^t2B_{\tau}dB_{\tau})$ is ($bw^*$)-continuous, subadditive and positively homogeneous. 
\end{proposition}
\begin{proof}
Clearly, since $\int_0^t2B_{\tau}dB_{\tau}=B^2_t-te$, it follows that the difference $B_t^2-\int_0^t2B_{\tau}dB_{\tau}$ is a positive multiple of $e$, i.e. a convex compact set. The conclusion then follows from  Theorem \ref{corollario2.5Gine}.
\end{proof}

The generalized Hukuhara difference can be expressed by means of the support functions as in \cite[Proposition 8]{ls2010} in the following way:
\begin{proposition}\label{classificazione}
Let $s(\cdot, A), s(\cdot, B)$ be the support functions of $A,B \in ck(X)$ and denote by
$s_1:=s(\cdot,A) - s(\cdot,B), s_2 = s(\cdot, B) - s(\cdot,A)$.
Then  only four cases may occur:
\begin{enumerate}[\rm \bf \ref{classificazione}.a)]
\item if $s_1,s_2$ are bw*-continuous and subadditive then $A \ominus_g B \in ck(X)$ and $A \ominus_g B$ is a singleton;
\item if only $s_1$ is bw*-continuous and  subadditive then equation {\rm  \ref{ghdiff} (i)} holds and $s(\cdot,C) :=s_1$;
\item if only $s_2$ is bw*-continuous and  subadditive then  equation {\rm  \ref{ghdiff} (ii)} holds and $s(\cdot,C) := s(\cdot, B) - s(\cdot,-A)$;
\item if none of them is bw*-continuous or subadditive then $A \ominus_g B$ does not exist.
\end{enumerate} 
\end{proposition}
\begin{proof}
The proof of the first statement is the same as in \cite[Proposition 8]{ls2010} since it depends only on the subadditivity of $s_i,i=1,2$ and the fact that in this case $A,B$ are each one a traslate of the other, the bw*-continuity of $s_1, s_2$ implies again by Theorem \ref{corollario2.5Gine}, that  $A \ominus_g B \in ck(X)$. \\
As to the second statement observe that by \cite[Theorem 2.3]{gh1985} there exists $C \in ck(X)$ such that $s_1 =s(\cdot,C)$, so
 (\ref{ghdiff}.i) is valid.\\
 In the third case the same can be done for $s_2$, so there exists $D \in ck(X)$ such that $s_2 =s(\cdot,D)$. Set now $C=-D$, the rest of the proof follows as in the quoted \cite[Proposition 8]{ls2010}. \\
Finally since  \cite[Theorem 2.3]{gh1985} is a necessary and sufficient condition that there exist no $C \in ck(X)$ such that $A = B + C$ or
$D \in ck(X)$ such that $B = A + D$ ,
so $A \ominus_g B$ does not exists.
\end{proof}

\noindent {\bf \large Acknowledgement} \\
The  paper was written while the second named author was visiting the 
 Departments of Mathematics and Computer Sciences of the  University  of Palermo and 
 Perugia (Italy) thanks to the Grant  Prot. N. U2014/000854 
of GNAMPA - INDAM (Italy).\\
The first and the last authors have been supported by University of Perugia -- Department of Mathematics and Computer Sciences-- Grant Nr 2010.011.0403, Prin "Metodi logici per il trattamento dell'informazione" and    "Descartes".\\
The second author has been supported by the Grant N. 87502 by N.R.F.\\
The third author was supported by F.F.R. 2013 Prof. Di Piazza - University of Palermo (Italy).

\Addresses
\end{document}